\documentclass[10pt]{amsart}

\usepackage{amsmath, amssymb, amscd}
\usepackage{eucal}
\usepackage{mathrsfs}
\usepackage[frame,cmtip,arrow,matrix,line,graph,curve]{xy}
\usepackage{graphicx}

\newtheorem{prop}{Proposition}[section]
\newtheorem{thm}[prop]{Theorem}
\newtheorem{lem}[prop]{Lemma}

\theoremstyle{remark}
  \newtheorem{rk}[prop]{Remark}
  
\theoremstyle{definition}

\numberwithin{equation}{section}

\begin{document}
\title[A topological characterization of omega-limit sets on dynamical systems] {A topological characterization of omega-limit sets on dynamical systems}
\author{Hahng-Yun Chu, Ahyoung Kim$^{\ast}$ and Jong-Suh Park}

\address{Hahng-Yun Chu, Department of Mathematics, Chungnam National University, 79, Daehangno, Yuseong-gu,
Daejeon 305-764, Republic of Korea}
\email{\rm hychu@cnu.ac.kr\ (H.-Y.Chu)}
\address{Ahoung Kim, Department of Mathematics, Chungnam National University, 79, Daehangno, Yuseong-gu,
Daejeon 305-764, Republic of Korea}
\email{\rm aykim111@hotmail.com\ (A.Kim)}
\address{Jong-Suh Park, Department of Mathematics, Chungnam National University, 79, Daehangno, Yuseong-gu,
Daejeon 305-764, Republic of Korea}
\email{\rm jspark@cnu.ac.kr,  jspark3141@gmail.com\ (J.-S.Park)}

\thanks{\it $\ast$ Corresponding author.}

\subjclass[2010]{37B25.}
\keywords{pseudo-orbit, limit sets, $\Omega$-limit sets, attracting, quasi-attracting.}

\begin{abstract}

In this article, we deal with several notions in dynamical systems. Firstly, we prove that both closure function and orbital function are idempotent on set-valued dynamical systems. And we show that the compact limit set of a connected set is also connected. Furthermore, we prove that the $\Omega$-limit set of a compact set is quasi-attracting.

\end{abstract}

\maketitle

\thispagestyle{empty} \markboth{Hahng-Yun Chu, Min-Young Kim and Jong-Suh Park} {A topological characterization of omega-limit sets on dynamical systems}


\section{Introduction}\label{sec: intro}

\smallskip

The theory for the notions of attractors and omega-limit sets is so important for the classical theory of dynamical systems. Conley\cite{Co78} introduced a topological definition of attractors for a flow on a compact metric space. Hurley\cite{Hu91, Hu92} obtained results which is related to the correspondence between attractors and Lyapunov functions on noncompact spaces.
Akin\cite{Ak93} and McGehee\cite{Mc92} obtained many properties of attractors in the set-valued dynamics.
Set-valued dynamical systems appear to be rather suitable for describing the global behavior of processes in optimal control dynamics and economic dynamics. Also the systems are used to describe multi-valued differential equations.
Chu et al.\cite{Ch05, CP05, Ch08} dealt with the notions of attractors, recurrences and stabilities in set-valued dynamical systems. The properties of set-valued dynamical systems have been investigated in several papers.\cite{Ch05,Ch08,CP05,Mc92,PKC07}

The concept of omega-limit set, arising from their ubiquitous applications in dynamical systems, is also an extremely used tool in the abstract theory of dynamical systems. Especially, the notion of omega-limit sets is much related to the notion of attractors. These notions are used to describe eventually the positive time behavior for dynamical systems.
A pseudo-orbit(chain) was firstly used by Bowen \cite{Bo75} and Conley \cite{Co78}. The notion is a very strong tool to understand important theories in the several fields of Mathematics and generates many results about the induced concepts, for example, chain transitive, chain recurrence, shadowing property and so on.
See \cite{Ao94, BCP86, BS70, Bl92, KP10, Op05,Pi99, Si75, Vr93}.

In this paper, we firstly focus on the properties of the special two functions in set-valued dynamical systems. And we pay specially close attention to limit sets, and also explain a kind of invariance for the limit sets and discuss the properties of attracting sets on locally compact metric spaces. Moreover, we consider the relationship between the notion of an $\Omega$-limit set and the notion of an quasi-attracting set. Here, the quasi-attracting set precisely means the intersection of attracting sets and then becomes the general version of the notion of attracting set.
More precisely, we show that the $\Omega$-limit set of compact set is quasi-attracting.

The paper is organized as follows.

In section $2$, we explain the elementary definitions for the proof of the main theorems. We focus on a closure function and an orbital function on the power set of the given space and show that the above two functions have idempotent property.

In section $3$, we study the invariance for a limit and show that the connectedness in the power set is invariant under the limit.
Next, we briefly sketch for the theories of attracting sets and quasi-attracting sets.
We also prove that the $\Omega$-limit set of a compact subset of $X$ becomes an quasi-attracting set.

\section{Set-valued dynamical systems}\label{sec: rec}

\smallskip

Let $(X,d)$ be a locally compact metric space. A {\textit{flow}} on $X$ is a continuous map $\pi : X \times \mathbb{R} \to X$ that satisfies the following group laws; for every $x \in X$, $\pi(x,0)=x$ and for every $t,s \in \mathbb{R}$, $x \in X$, $\pi(\pi(x,t),s)=\pi(x,t+s)$.
For a convenience, we briefly write $x \cdot t=\pi(x,t)$.
For any $x \in X$, we define an \textit{orbit} of $x$ to be the subset $\{x \cdot t\mid t \in \mathbb{R}\}$ of $X$ which is denoted by $O(x)$.
We say that a subset $Y$ of $X$ is \textit{positively invariant {\rm{(}}invariant{\rm{)}}} under $\pi$ if $Y \cdot \mathbb{R}^{+}=Y (Y \cdot \mathbb{R}=Y)$.

Next, we introduce the subsets of $X$ which are related to the eventual orbit of a point under the flow $\pi$. For $x \in X$, the {\textit{limit set}} of $x$, denoted by $\Lambda^{+}(x)$, is defined by $$\Lambda^{+}(x):=\bigcap_{~ t\geq 0} \overline {x\cdot [t,\infty)}.$$
The limit set of $x$ has a major role in Conley's theory, and for its basic properties we refer to \cite{Co78,Di08}.
For  $x \in X$, we also call the {\textit{first prolongational limit set}} and {\textit{first prolongational set}} of $x$ as defined, respectively, by
\begin{eqnarray*}
&&J^{+}(x):=\bigcap_{U\in N(x), t\geq 0}\overline{U \cdot [t,\infty)},\\
&&D^{+}(x):=\bigcap_{U\in N(x)}\overline{U\cdot \mathbb{R}^{+}},
\end{eqnarray*}
where $N(x)$ is the set of all neighborhoods of $x$.

\smallskip
In \cite{BCP86}, Bae, Choi and Park studied limit sets and prolongational sets in topological dynamics. Next remark is immediately proved from the definitions.

\smallskip

\begin{rk} For $x \in X$, the following equivalences hold.

\begin{itemize}
\item[(1)] $y \in \Lambda^{+}(x)$ if and only if there is a sequence $\{t_{n}\}$ in $\mathbb{R}^{+}$ with $t_{n}\rightarrow \infty$ such that $x\cdot t_{n}\rightarrow y$.
\item[(2)] $y \in J^{+}(x)$ if and only if there are a sequence $\{x_{n}\}$ in $X$ and a sequence $\{t_n\}$ in $\mathbb{R}^{+}$ such that $x_{n}\rightarrow x$, $t_{n}\rightarrow \infty$ and $x_{n} \cdot t_{n}\rightarrow y$.
\item[(3)]$y \in D^{+}(x)$ if and only if there are a sequence $\{x_n\}$ in $X$ and a sequence $\{t_{n}\}$ in $\mathbb{R}^{+}$ such that $x_{n}\rightarrow x$ and $x_{n}t_{n}\rightarrow y$.
\end{itemize}
\end{rk}

\smallskip

Let $\Gamma : X \rightarrow 2^{X}$ be a function and $A \subseteq X,$ then we may canonically define the new mapping from the power set of $X$ to the same set as follows,
$$
\Gamma(A)=\cup_{x \in A} \Gamma(x).
$$
We define the composition $\Gamma^2 = \Gamma \circ \Gamma$ given by $\Gamma^2(x) = \Gamma(\Gamma(x))= \cup_{y \in \Gamma(x)} \Gamma(y)$, so we can define naturally the iteration $\Gamma^n : X \rightarrow 2^{X}$ inductively by $\Gamma^{1}(x)=\Gamma(x)$ and $\Gamma^{n}(x)=\Gamma(\Gamma^{n-1}(x))$.
So the trajectory for the function $\Gamma$ can be expressed by the union of the iteration $\Gamma^{n}$.
For a family of functions $\Gamma_{i} : X \rightarrow 2^{X}(i \in I)$, we give the new map $\cup_{i \in I} \Gamma_{i} : X \rightarrow 2^{X}$ defined by $(\cup_{i \in I} \Gamma_{i})(x)=\cup_{i \in I} \Gamma_{i}(x)$.

Let $\mathcal{P}$ be the set of all functions from $X$ to its power set $2^{X}$ and let $\Gamma \in \mathcal{P}$. We define the mappings $D$ and $S$ from the set $\mathcal{P}$ to itself given by, for every $x \in X$,
$$(D\Gamma)(x)=\cap_{U \in N(x)}\overline{\Gamma(U)} \ \ {\rm and} \ \ (S\Gamma)(x)=\cup_{n=1}^{\infty} \Gamma^{n}(x).$$
We call $D$ a \textit {closure funtion} for $\Gamma$ and $S$ a \textit {orbital function} for $\Gamma$ defined on $\mathcal{P}$.

\smallskip

\begin{rk}\label{2.2}\cite{CKP13}
  Let $\Gamma :X \to 2^{X}$ be a mapping and $x \in X.$ Then $D\Gamma(x)$ is the set of all points $y \in X$ with the property that there exist sequences $(x_{n})$ and $(y_{n})$ in $X$ with $y_{n} \in \Gamma(x_{n})$ such that $x_{n}\rightarrow x$, $y_{n}\rightarrow y.$ Furthermore, $S\Gamma(x)$ is the set of all points $y \in X$ such that there is a finite subset $\{x_{1}, \cdots , x_{k}\}$ of $X$ with the properties that $x_1=x, x_k=y$ and $x_{i+1} \in \Gamma(x_{i})$, $i=1,\cdots,k-1$.
\end{rk}

\smallskip

The new mappings have interesting properties, especially the iterations of the mappings are just the original mappings.

\smallskip

\begin{thm}
A closure function $D$ is idempotent and so is an orbital function $S$, that is, $D^{2}=D$ and $S^{2}=S$.
\end{thm}

\begin{proof}
Firstly, we show that $D^{2}\Gamma(x)$ is contained in $D\Gamma(x)$. From the definition of the mapping, we directly consider the equalities
\begin{eqnarray*}
D^{2}\Gamma(x)
&=& (D(D\Gamma))(x)\\
&=& \cap_{U\in N(x)} \overline{D\Gamma (U)}\\
&=& \cap_{U\in N(x)} \overline{\cup_{y \in U}D\Gamma (y)}\\
&=&\cap_{U\in N(x)} \overline{\cup_{y \in U} \cap_{V\in N(y)} \overline{\Gamma (V)}}.
\end{eqnarray*}
Here, for each $y \in U$, we also obtain the following inclusion
$$
\cap_{V\in N(y)} \overline{\Gamma (V)} \subseteq \overline{\Gamma(U)}.
$$
Thus we have that
$$
\cup_{y \in U} \cap_{V\in N(y)} \overline{\Gamma (V)} \subseteq \overline{\Gamma(U)},
$$
and so
\begin{eqnarray*}
\cap_{U\in N(x)} \overline{\cup_{y \in U} \cap_{V\in N(y)} \overline{\Gamma (V)}}
&\subseteq& \cap_{U\in N(x)} \overline{\Gamma (U)} = D\Gamma(x).\\
\end{eqnarray*}

Conversely, we look at the opposite direction of the proof to get the equality.
For every neighborhood $V$ of $y$, we get $\Gamma(y) \subseteq \cap_{V \in N(y)}\overline{\Gamma(V)}$. By the definitions, we obtain that
$$
\overline{\Gamma(U)} = \overline{\cup_{y \in U}\Gamma(y)} \subseteq \overline{\cup_{y \in U} \cap _{V \in N(y)} \overline{\Gamma(V)}},
$$
so we conclude that
$$
D\Gamma(x)=\cap_{U \in N(x)}\overline{\Gamma(U)} \subseteq \cap_{U \in N(x)}\overline{\cup_{y \in U} \cap _{V \in N(y)} \overline{\Gamma(V)}} = D^2 \Gamma(x).
$$
Then we have the first equality of this proof.

Next, we consider the case for the mapping $S$.
By the definitions of the mappings, we get that
\begin{eqnarray*}
S^{2}\Gamma(x)
&=& (S(S\Gamma))(x) \\
&=& \cup_{n=1}^{\infty}(S\Gamma)^n(x) \\
&=& (S\Gamma)(x) \cup (S\Gamma)^2(x) \cup \cdots ,
\end{eqnarray*}
for every $x \in X$.
It is clear that $(S^2\Gamma)(x) \supseteq (S\Gamma)(x)$.
Now we prove another inclusion $S(S\Gamma)(x) \subseteq (S\Gamma)(x)$.
Firstly, we obtain that
\begin{eqnarray*}
(S\Gamma)^{2}(x)
&=& (S\Gamma)((S\Gamma)(x)) \\
&=& \cup_{k=1}^{\infty} \Gamma^k(S\Gamma(x)) \\
&=& \cup_{y \in S\Gamma(x)} \cup_{k=1}^{\infty} \Gamma^k(y) \\
&=& \cup_{y \in \cup_{n=1}^{\infty}\Gamma^n(x)} \cup_{k=1}^{\infty} \Gamma^k(y) \\
&=& \cup_{n=1}^{\infty} \cup_{y \in \Gamma^n(x)} \cup_{k=1}^{\infty} \Gamma^k(y) \\
&=& \cup_{n=1}^{\infty} \cup_{k=1}^{\infty} \cup_{y \in \Gamma^n(x)} \Gamma^k(y) \\
&=& \cup_{n=1}^{\infty} \cup_{k=1}^{\infty} \Gamma^k(\Gamma^n(x)) \\
&\subseteq& \cup_{n=1}^{\infty}\Gamma^n(x) \\
&=& S\Gamma(x).
\end{eqnarray*}
Thus we have $(S\Gamma)^2(x) \subseteq S\Gamma(x)$.
Using the induction, firstly we assume that
$(S\Gamma)^l(x) \subseteq (S\Gamma)(x)$ for every positive integer $l \leq m$.
Then, by the properties of the mappings, we have that
\begin{eqnarray*}
(S\Gamma)^{m+1}(x)
&=& (S\Gamma)((S\Gamma)^m(x))  \\
&\subseteq&  (S\Gamma)((S\Gamma)(x)) \\
&\subseteq& S\Gamma(x).
\end{eqnarray*}
So $(S\Gamma)^n(x) \subseteq (S\Gamma)(x)$, for every positive integer $n$.
Then, $\cup_{n=1}^{\infty}(S\Gamma)^n(x) \subseteq (S\Gamma)(x)$ and so $S(S\Gamma)(x) \subseteq S\Gamma(x)$ for all $x$.
Hence we have $S^2\Gamma = S\Gamma$, as desired.
\end{proof}

\smallskip

We recall the notions of chains and $\Omega$-limit sets in \cite{Co78} for details. Let $x,y$ be elements of $X$ and $\epsilon$, $t$ positive real numbers. An {\it $(\epsilon,t)$-chain} from $x$ to $y$ means a pair of finite sequences $x=x_{1},x_{2},\cdots,x_{n},x_{n+1}=y$ in $X$ and  $t_{1},t_{2},\cdots,t_{n}$ in
$\mathbb{R}^{+}$ such that $t_{i}\geq t$ and $d(x_{i} \cdot t_{i},x_{i+1})\leq \epsilon$ for all $i=1,2,\cdots,n$. Define a relation $R$ in $X \times X$ given by $xRy$ means, for every $\varepsilon>0$ and $t>0$, there exists $(\epsilon,t)$-chain from $x$ to $y$. We also denote $(x,y) \in R$ by $xRy$.

For $x \in X$, we define the {\it $\Omega$-limit set} of $x$ by $\Omega(x)= \{ y \in X : (x,y) \in R \}$. We also canonically define a map $\Omega : X \to 2^{X}$ given by $x \mapsto \Omega(x).$

In \cite{Co78}, Conley investigated the several notions of topological dynamics in a compact metric space. He proved that the chain relation $R$ is closed and transitive on $X$ and also showed that if $(x,y) \in R$ and $(s_1, s_2) \in \mathbb{R}^+ \times \mathbb{R}^+$, then $(x\cdot s_{1},y \cdot s_{2}) \in R$. We observe that $\Omega(x)$ is a closed invariant subset of a compact metric space $X$ and $J^{+}(x)\subseteq \Omega(x)$(see \cite[p.36]{Co78} and \cite[p.2721]{Di08}).

\medskip


\medskip

\section{Attracting sets and quasi-attracting sets}\label{sec: chain2}

\medskip

In this section, we investigate the properties of the limit sets, attracting sets and quasi-attracting sets in a locally compact space $X$.

For a subset $Y$ of $X$, we define the {\it limit set} of $Y$ by
$$
\omega(Y)=\bigcap_{~ t\geq 0} \overline {Y\cdot [t,\infty)}.
$$
Note that $\omega(Y)$ is a maximal invariant subset in $Y\cdot [0,\infty)$ and is generally larger than $\Lambda^+(Y)=\bigcup_{x\in Y}\Lambda^+(x)$.

A positively invariant closed subset $A$ of $X$ is called an {\it attracting set} if $A$ admits a neighborhood $U$ such that $\omega(U) \subseteq A$. A closed set $A$ which is the intersection of attracting sets will be called a {\it quasi-attracting set}.

It is easy to see that if $A$ is a (quasi-)attracting set, so is $A\cdot t$, for $t \in \mathbb{R}$.
Let $\{B_i\}_{i \in I}$ (here, $I$ is some index set) be the family of quasi-attracting sets, then $\bigcap_{i\in I}B_i$ is also a quasi-attracting set.
We note that a quasi-attracting set is just positively invariant and in general (quasi-)atracting set need not invariant. If an (quasi-)attracting set is invariant, the set is called an (quasi-)attractor in the sense of Conley(see \cite{Co78}); that is, an invariant attracting set $A$ is an attractor. Similarly, we define a quasi-attractor by the intersection of attractors.

Next, we introduce the second main theorem which is related to an invariance. More precisely speaking, the connectedness is invariant under the notion of limit.

\smallskip

\begin{thm} Let  $Y$ be a connected subset of $X$. If the limit set $\omega(Y)$ of $Y$ is compact, then the limit set is connected.
\end{thm}

\begin{proof} Suppose the contrary of the conclusion. Then we can choose the separations $A$ and $B$ of $\omega(Y)$. Note that the subsets $A$ and $B$ of $\omega(Y)$ are mutually disjoint clopen subsets in $\omega(Y)$. From the compactness of $\omega(Y)$, the disjoint sets are also compact in $X$. Then there exist disjoint neighborhoods $U$ and $V$ of $A$ and $B$, respectively.
Put $A_k:=\overline{Y \cdot [k,\infty)}$.
\begin{lem}\label{number} Let $\{A_n\}$ be a sequence of closed, connected subsets of $X$
with satisfying the properties that
\begin{itemize}
\item[(1)]$A_1\supseteq A_2\supseteq \cdots ,$
\item[(2)]$\bigcap_{n=1}^{\infty}A_n$ is a nonempty compact subset of $X$,
\end{itemize}
then for an arbitrary neighborhood $U$ of $\bigcap_{n=1}^{\infty}A_n$, there is a natural number $n$ such that $A_n \subseteq U$.
\end{lem}

\begin{proof} First, let $A$ be the subset $\bigcap_{n=1}^{\infty}A_n$ of $X$. Suppose the contrary of the conclusion, that is, there exists a neighborhood $U$ of $\bigcap_{n=1}^{\infty}A_n$ such that $A_n\nsubseteq U$ for each natural number $n$. Owing to the fact that $A$ is the compact subset of a locally compact space $X$, there exists a neighborhood $V$ of $A$ such that $\overline{V} \subseteq U$ and $\overline{V}$ is compact. From the assumption, for each $n$, $A_n$ is not contained in $U$, so $A_n \cap V^c$ is nonempty.
Also, since $A_n \cap V$ contains the set $A$, it is nonempty. From the connectedness of $A_n$, $A_n$ intersects to the boundary of $V$. So, we can choose an element $x_n$ of $A_n \cap \partial V$. Since the boundary $\partial V$ is also compact, there exists a convergent subsequence of $\{x_n\}$ in $\partial V$. Without loss of generality, we can assume that the sequence $\{x_n\}$ converges to a point $x$ in $\partial V$.
Let $k$ be an arbitrary natural number. For a natural number $n$ which is bigger than $k$, $x_n$ is contained in $A_k$. So, the limit point $x$ is also an element of $\overline{A_k}$. Since $A_k$ is closed, $x \in \bigcap_{k=1}^{\infty}A_k=A$. Hence, we derive a contradiction from the fact that $V$ is the neighborhood of $A$. Therefore, we can find a natural number $n$ such that $A_n \subseteq U$, which completes the proof.
\end{proof}
By Lemma \ref{number}, there exists a natural number $n$ such that $A_n \subseteq U \cup V.$ Since $A_n$ is connected, either $A_{n}\subseteq U$ or $A_n \subseteq V$. Without loss of generality, we can suppose that $A_n \subseteq U$. If $k \geq n$, we obtain that $A_k \subseteq A_n \subseteq U$ and so $\cap_{k=n}^{\infty} A_k= \omega(Y) \subseteq U$. This is a contradiction.
Thus, $\omega(Y)$ is connected.
\end{proof}

\smallskip

Now, for arbitrary positive real numbers $\varepsilon$ and $t$, we define the set $P_t(M,\varepsilon)$ by
$$P_t(M,\varepsilon):=\{y \in X \mid {\rm{there \ is\ an}} \ (\varepsilon, t){\rm{-chain}} \ {\rm{from}} \ x \ {\rm{to}}\  y \ {\rm{for}\  some} \  x \in M \}.$$

In the following lemma, we show that the $\Omega$-limit set $\Omega(M)$ of compact set $M$ is represented by the intersection of the above subsets. Actually, this representation play an important role in the proof of Theorem \ref{cpt}.

\smallskip

\begin{lem} \label{ome} Let $M$ be a compact subset of $X$. Then $$\Omega(M)= \bigcap_{ \ \varepsilon,t>0}P_t(M,\varepsilon).$$
\end{lem}

\begin{proof}
Firstly we recall the equality $\Omega(M)=\cup_{x \in M} \Omega(x)$. From the equality, we can easily prove the inclusion $\Omega(M) \subseteq \bigcap_{ \ \varepsilon,t>0}P_t(M,\varepsilon)$.

Conversely, we let $y \in \bigcap_{ \ \varepsilon,t>0}P_t(M,\varepsilon)$. Then, for each positive integer $n$, since $y$ is an element of $P_n(M,\frac{1}{n})$, there exist an element $x_n$ of $M$ and a $(\frac{1}{n}, n)$-chain from $x_n$ to $y$
$$
\{x_n=x_1^n, x_2^n, \cdots, x_{m_n}^n, x_{m_n+1}^n=y ; t_1^n, t_2^n, \cdots, t_{m_n}^n\}.
$$
Since $M$ is compact, the sequence $\{x_n\}$ in $M$ has a convergent subsequence.
Without loss of generality, we can assume that the original sequence $\{x_n\}$ converges to some point $x$ in $M$. For any $\varepsilon>0$ and $t>0$, there exists a positive real number $\delta$ such that
if $d(x,z)< \delta$, then $d(x\cdot t, z\cdot t)< \varepsilon$.
We can take a positive integer $n$ such that
$$
d(x, x_n)<\delta, n>2t \ {\rm{and}} \  \frac{1}{n}<\varepsilon.
$$
Since $d(x, x_n)<\delta$, we have $d(x\cdot t, x_n \cdot t)<\varepsilon.$ Thus the following sequence
$$
\{x, x_n\cdot t, x_2^n, \cdots, x_{m_n}^n, x_{m_n+1}^n=y; t, t_1^n-t, t_2^n, \cdots, t_{m_n}^n\}
$$
is $(\varepsilon, t)$-chain from $x$ to $y$.
Hence $y$ is an element of $\Omega(M)$, which completes the proof.
\end{proof}

\smallskip

In \cite[p.2724]{Di08}, Ding showed that the set $P_t(M,\varepsilon)$ is open and proved the inclusion $\overline{P_t(M,\varepsilon)\cdot [t,\infty)} \subseteq P_t(M,\varepsilon)$. Now put $A_{\varepsilon,t}:=\overline{P_t(M,\varepsilon)\cdot [t,\infty)}$. Then he got that the set $P_t(M,\varepsilon)$ is an open neighborhood of $A_{\varepsilon,t}$ and furthermore, $A_{\varepsilon,t}$ is a positively invariant closed attracting set.

\begin{lem}\label{4.7}
Let $P_t(M,\varepsilon)$ and $A_{\varepsilon,t}$ be the same as notations in the above statements. Then we have $$\bigcap_{\varepsilon,t>0}P_t(M,\varepsilon)=\bigcap_{\varepsilon,t>0}A_{\varepsilon,t}.$$
\end{lem}

\begin{proof} From the Ding's results, it is obvious that the set $\bigcap_{\varepsilon,t>0}A_{\varepsilon,t}$ is contained in the intersection $\bigcap_{\varepsilon,t>0}P_t(M,\varepsilon)$.

To show the converse, firstly, let $y$ be an element of $\bigcap_{\varepsilon,t>0}P_t(M,\varepsilon)$.
For an arbitrary positive real numbers $\varepsilon$ and $t$, we can choose a positive real number $t^{'}$ larger than $2t$. We note that the action $y \cdot (-s)(s>0)$ is continuous.
Let $\delta$ be an arbitrary positive real number. Using the continuity of the action, we can choose a positive real number $\delta^{'}$ such that if $d(y,y^{'})< \delta^{'}$ then
\begin{eqnarray}\label{*}
d(y\cdot (-t),y^{'}\cdot (-t))< \delta.
\end{eqnarray}

Put $\varepsilon^{'}:=min(\delta^{'}, \varepsilon)$. Since $y$ is an element of $P_{t^{'}}(M,\varepsilon^{'})$, there exists $(\varepsilon^{'},t^{'})$-chain from $x_{0}$ to $y$, say $\{x_{0}, x_1, \cdots, x_m, x_{m+1}=y; t_{0}, t_1, \cdots, t_m\}$, for some $x_{0}$ in $M$. Since $t^{'}$ is larger than $2t$, we can construct the new $(\varepsilon,t)$-chain from $x_{0}$ to $x_m \cdot (t_m-t)$ as follows,
$$
\{x_{0}, x_1, \cdots, x_m, x_{m}\cdot (t_m-t); t_{0}, t_1, \cdots, t_m-t\}.
$$
Thus we have that $x_m\cdot (t_m-t)$ is an element of $P_{t}(M,\varepsilon)$.
Note that the inequalities $d(y, x_m\cdot t_m)< \varepsilon^{'} \leq \delta^{'}$. By (\ref{*}), we obtain that $d(y\cdot (-t),x_m\cdot (t_m-t))< \delta$. Since $\delta$ is arbitrary, it yields that $y\cdot (-t) \in \overline{P_{t}(M,\varepsilon)}$. Then we gain the following inclusions
\begin{eqnarray*}
y=(y\cdot (-t)) \cdot t
&\in& \overline{P_{t}(M,\varepsilon)}\cdot t \\
&\subseteq& \overline{P_{t}(M,\varepsilon)\cdot t} \\
&\subseteq&  \overline{P_{t}(M,\varepsilon)\cdot [t,\infty)},
\end{eqnarray*}
for arbitrary  positive real numbers $\varepsilon,t$.
Therefore $y$ is an element of $\bigcap_{ \ \varepsilon,t>0}
A_{\varepsilon,t}$.
\end{proof}

\smallskip

To prove the theorem \ref{cpt}, we also need a basic property for $\Omega$-limit sets as follows.
\smallskip

\begin{lem}\cite{CKP13}\label{omega} If $p_{n} \to p, q_{n} \to q$ and $q_{n} \in \Omega(p_{n})$, then $q \in \Omega(p)$.
\end{lem}

\smallskip

In \cite{Di08}, Ding proved that the image of closed set under the chain prolongation $P$ is quasi-attracting. Especially, in the proof of the theorem 4.6 \cite[p.2724]{Di08}, he claimed that the image of closed set under the chain prolongation $P$ is also closed using the theorem 3.4 \cite[p.2721]{Di08}.
We now mention another result for quasi-attracting sets in the next theorem. In detail, we prove that the $\Omega$-limit set of a compact subset of $X$ becomes the intersection of attracting sets, that is, a quasi-attracting set.

\smallskip

\begin{thm}\label{cpt} For a compact subset $M$ of $X$, $\Omega(M)$ is quasi-attracting.
\end{thm}

\smallskip

\begin{proof} First of all, we show that $\Omega(M)$ is a closed subset. Let $y$ be an element of $\overline{\Omega(M)}$. Then there exists a sequence $\{y_n\}$ in $\Omega(M)$ such that $\{y_n\}$ converges to $y$. Thus, for every $n$, there exists a point $x_n$ in $M$ such that $y_n \in \Omega(x_n)$. By the compactness of $M$, the sequence $\{x_n\}$ has a convergent subsequence. Without loss of generality, we can assume that the original sequence $\{x_n\}$ converges to some point $x$ in $M$.
By the lemma \ref{omega}, $y$ is an element of $\Omega(x)$ and thus, $\Omega(M)$ is closed.

By combining Lemma \ref{ome} and \ref{4.7}, we reach the following equalities
$$
\Omega(M)=\bigcap_{\varepsilon,t>0}P_t(M,\varepsilon)=\bigcap_{\varepsilon,t>0}A_{\varepsilon,t}.
$$
In conclusion, the $\Omega$-limit set becomes the intersection of the attracting sets, that is, quasi-attracting set. This completes the proof.
\end{proof}

\smallskip





\end{document}